\documentclass[12pt,oneside,reqno]{amsart}
\usepackage{graphicx}
\usepackage[breaklinks=true,colorlinks=true,linkcolor=blue,urlcolor=blue,citecolor=blue]{hyperref}

\usepackage[letterpaper, portrait, margin=1in]{geometry}
\usepackage{amssymb}
\newcommand{\R}{{\mathbb R}}
\newcommand{\Z}{{\mathbb Z}}

\newcommand{\C}{{\mathbb C}}

\newtheorem{thm}{Theorem}[section]
\newtheorem{coro}[thm]{Corollary}
\newtheorem{definition}[thm]{Definition}
\newtheorem{lem}[thm]{Lemma}
\newtheorem{pro}[thm]{Proposition}
\newtheorem{rem}[thm]{Remark}

\def\ii{\'\i}

\begin{document}

\title[Estimate of level sets for KPP equations with delay]{An estimation of level sets for non local KPP equations with delay}

\author[Benguria]{Rafael Benguria}
\address{Instituto de F\ii sica,  Facultad de F\ii sica
P. Universidad Cat—\'olica de Chile 
Casilla 306, Santiago 22, Chile }
\email{rbenguri@fis.puc.cl}

\author[Solar]{Abraham Solar}
\address{Instituto de F\ii sica,  Facultad de F\ii sica
P. Universidad Cat—\'olica de Chile 
Casilla 306, Santiago 22, Chile }
\email{asolar@fis.puc.cl}

\maketitle

\begin{abstract} 
We study the large time asymptotic behavior of the solutions of the linear parabolic equation with delay
$(*)$: $u_{t}(t,x) = u_{xx}(t,x) - u(t,x) + \int_{\mathbb{R}} k(x-y) \, u (t-h, y)\, dy$,  $x \in \R$, $\ t >0$, and $0\leq k(x) \in L^1(\R)$. 
As an application we get estimates on the measure of level sets of non local KPP type equations with delay.
For this type of nonlinear equations we prove that, in contrast with the classical case, the solution to the initial value 
problem with data of compact support may not be persistent.

\end{abstract}


\vspace{2pc}
\newpage
\section{Introduction}

In this manuscript we consider the following delayed partial differential equation  
\begin{eqnarray}\label{dle0}
u_t(t,x) & =u_{xx}(t,x)+ mu_x(t,x)+pu(t,x) \nonumber \\ &+\int_{\mathbb{R}} k(x-y) \, u (t-h, y)\, dy
\end{eqnarray}
for $t>0$, $x\in\R$, and  $h>0$. 
Here, $m,p\in\R$ and $ 0\leq k(x) \in L^1(\R)$.

\bigskip
\noindent
Although the study of the solutions of (\ref{dle0}) is interesting in itself (see, e.g.,  \cite{BP,D, FN,KSch}), 
it has special relevance in the description of the  dynamics of non--local reaction difusion equations.
In particular, the solutions of  (\ref{dle0}) are relevant in the study of the following nonlinear, nonlocal,  evolution equation with delay, 
\begin{eqnarray}\label{nlne}
 v_{t}(t,x) = v_{xx}(t,x) -  v(t,x) + \int_{\R}k_0(y)g(v(t-h,x-y))dy,
\end{eqnarray}
for $t>0$, $x\in\R$, and  $h>0$. 
In (\ref{nlne}) $g$ is a Lipschitz  function which has exactly two fixed points: 0 and  $\kappa$.  This equation models the dynamics of a broad class of populations \cite{SWZ, TZ,WLR,TAT,MOZ, YCW, GPT,YZ}. Here $g$ stands for 
the birth rate of a given population, while $v(t,x)$ is the measure of sexually mature adults in that population, at the point $x$ at time $t$. Important in what follows is the parameter $h$, which 
is the total time spent by an individual from birth until becoming a sexually mature adult. Finally, the kernel $k_0$  takes into account the (non local) interaction between individuals. Of course, 
when $h=0$, $g=2u-u^2$ and $k_0(y)=\delta(y)$ the equation (\ref{nlne}) reduces to the  classical model of Kolmogorov, Petrovskii and Piskunov \cite{KPP}. 

\bigskip
\bigskip
\noindent
Equation (\ref{nlne}) was introduced by So, Wu and Zou \cite{SWZ}  to describe the behavior of poulations with age structure (see also,  \cite[Subsection 4.1]{TZ} and \cite[Section 5]{GPT})
There are many situations related with population dynamics where the age structure of the population matters. One of the most studied models is the one by Nicholson, introduced in 1954
to describe the competition for food in laboratory populations of blowflies {\it Lucilia Caprina}. In the Nicholson model  $g(u)=ue^{-u}$ 
(see e.g.,  \cite{CMYZ, HMW, MOZ,S2,YCW,YZ} and references therein).

\bigskip
\bigskip
\noindent
As \textcolor{red}{it} is well known in the classical reaction diffusion equations without delay, when the reaction term satisfies the so called KPP condition (i.e., $g(u) \le g'(0) u$), the evolution of the disturbances 
is characterized by the dynamics of the linearized equation. The same situation occurs in our case. In fact, 
the linear equation (\ref{dle0}) plays a crucial role in establishing the asymptotic behavior of perturbations for the so called {\it semi--wavefronts} (see, e.g., \cite[Theorem 18]{GPT}). 
These are nonnegative bounded solutions of the form $v(t,x)=\phi_c(x+ct), \phi_c:\R\to\R_{+}$, which propagate with speed $c$, and such that  either $\phi_c(-\infty)=0$ or $\phi_c(\infty)=0$. 
Moreover, if in the first case one additionally has,  $\phi(+\infty)=\kappa$ or in the second case  $\phi_c(-\infty)=\kappa$, the solutions are called {\it wavefronts}. 
In the context of population dynamics wavefronts model the invasion, at constant speed, of one species over a given habitat. In fact, under the  KPP condition, 
the perturbation of semi-wavefronts of (\ref{nlne}) are approximated by the solutions to (\ref{dle0}), see \cite{CMYZ,HMW,LLLM, MOZ,S2}.

 \bigskip
\bigskip
\noindent
The study of the asymptotic propagation speeds and the existence of monotone wavefronts for (\ref{nlne}) when $k_0$ is not symmetric (i.e., when $k_0(-x) \neq k_0(x)$) dates back to Weinberger \cite{W}.
Weinberger's result was later extended in a more abstract setting  by Liang and Zhao \cite{LZ}. 
Recently, Yi and Zou \cite{YZ}, using lower and upper monotone semi--flows, extended the results of Weinberger to the case of non monotone $g$'s. 
For results on the existence of semi--wavefronts of (\ref{nlne}) under the  KPP condition see, e.g.,  \cite[Theorem 4.4 ]{YZ} and  \cite[Theorem 5]{TAT},
 while for semi-wavefronts possibly non wavefronts see  \cite[Theorem 4]{TAT} and \cite[Theorem 18]{GPT}. From \cite[Theorem 18]{GPT}  the existence of 
 two critical speeds $c_*^-$ y $c^+_*$ follows.  These values for the speed are critical in the sense that if $c\geq c^+_*$ (or $c\leq c_*^-$) then there is a semi-wavefront with 
 speed $c$ such that  $\phi_c(+\infty)=0$ ($\phi_c(-\infty)=0$, respectively) and if  $c\in(c_*^-,c_*^+)$ then there are no semi-wavefronts with speed $c$. 
 Semi-wavefronts with critical speed $c_*^-$ or $c^+_*$ are called {\it critical semi -wavefronts} and they  have the following asymptotic behavior 
 $\lim_{z\to \pm\infty}\phi_{c^{\pm}_*}(z+z')ze^{-\lambda^{\pm}_*z}=-1$ for some $\lambda^{\pm}_*=\lambda^{\pm}_*(h)$ satisfying  $\pm\lambda^{\pm}_*> 0$ and $z'\in\R$ (see \cite[Theorem 3]{AGT}). 
 
 \bigskip
 \bigskip
 \noindent
 In the context of semi-wavefronts, we take $m=c$, $p=-1$ and $k(x)=g'(0)k_0(x-mh)$ in (\ref{dle0}). One of the goals of this paper is to give an optimal rate of convergence 
 for the approximation to semi-wavefronts given by equation (\ref{dle0}).  More precisely, the stability of wavefronts of (\ref{nlne}) implies an approximation of the semi-wavefronts by solving (\ref{nlne}) 
 as an inhomogenous linear PDE  on the time intervals $[0,h], [h, 2h], \dots$ for an appropriate initial data. In order to carry out this iterative scheme, the knowledge of the the asymptotic behavior of the solutions
 of (\ref{dle0}) is crucial. 
 
\bigskip
\noindent
 In this respect, recent investigations have obtained mainly results {on the stability to wavefronts} for small perturbations \cite{CMYZ} {to (\ref{nlne})} 
 and for global perturbations \cite{S} when $k_0=\delta$. It is worth pointing out that using this method one can approximate non monotone wavefronts uniformly on the real line. Moreover, 
one can also approximate  asymptotically periodic semi--wavefronts uniformly on any compact set of the real line 
(see, \cite[Theorem 3]{S}). {For $g$ monotone and $k_0$ a heat kernel in (\ref{nlne}), the authors in  \cite{MOZ} proved that the approximation is $O(t^{-{1/2}})$ 
for critical wavefronts and $O(e^{-\epsilon t})$, some $\epsilon>0$, for non-critical wavefronts. For more general $g$ but $k_0=\delta$,  numerical solutions are exhibited in \cite{CMYZ}
without providing an appropriate rate of convergence to the wavefront. Then, in \cite{S2} for $g$ (not necessarily monotone) and $k_0$ general,  the rate of convergence  $O(t^{-1/2})$  
is obtained for certain `hard perturbation'  and for `soft perturbation'  this rate can be exponentially improved. In our main result in Section 3, we prove that the convergence rate must be $o(t^{-1/2})$} for critical semi-wavefronts.

\bigskip
\noindent
In order to describe our main results concerning the solutions of (\ref{dle0}), embodied in Theorem 1.1  
below, we need some preliminary definitions.
Consider the functions, 
 \begin{eqnarray}\label{esw}
 q_1(z)=-z^2-mz-p\quad \hbox{and} \quad q_2(z)=\int_{\R}k(y)e^{-zy}dy,  
 \end{eqnarray}
 for which we assume the following hypothesis 

\bigskip
\noindent
 {\rm \bf(K)} The function  $q_2$ is defined on a maximal open interval   $(a,b)\ni 0$, for some real extended  
 \indent numbers $a$ and $b$, and 
$$
0<-p<\int_{\R}k(y)dy.
$$ 

\bigskip
\noindent
As is shown in \cite[Lemma 22]{GPT},  the condition {\bf (K)} and the convexity  of the function
$$
E_m(z):=-q_1(z)+q_2(z)=z^2+mz+p+\tilde{q}e^{zmh}\int_{\R}\tilde{k}(y)e^{-zy}dy,
$$
where $\tilde{q}:=||k||_{L^1(\R)}$ and $\tilde{k}(y):=k(y-mh)/\tilde{q}$ (following the structure of \cite[Lemma 22]{GPT}), imply that there exist numbers
 $m_-<m_+$ such that for every $m\in\mathcal{M}:=(-\infty,m_-]\cup[m_+,+\infty)$ the functions  $q_1$ y $q_2$ either  have exactly  two crossings or exactly one and if $m\in(m_-,m_+)$ then $q_1$ and $q_2$ 
 have no crossings. 

\bigskip
\noindent 
Now, in order to state  our main result we need the following tangential property
\vspace{3mm}

\noindent
{{\rm \bf(T)} For each $m\in\mathcal{M}$ there is a number $\gamma_m$ in such a way that  $q_1-\gamma_m$ and  $e^{h\gamma_m}q_2$ are 
 \indent  tangent at a  point  $z=z_m\in(a,b)$}
\vspace{3mm}

\noindent
{Note that  under the hypothesis  {\bf (T)}, the functions $q_1$ and $q_2$ have at least one crossing, and therefore necessarily $\gamma_m\geq 0$. 
Also, we note that for each $z\in (a,b)$ the function $E^*(\epsilon)=-q_1(z)+\gamma_m-\epsilon+q_2(z)e^{(\gamma_m+\epsilon)h}$ is 
monotone so that the numbers $\gamma_m$ and $z_m$ can be computed as the only solution $(\gamma,z)=(\gamma_m, z_m)$ to  the tangential equations 
 \begin{eqnarray}\label{t1}
 z^2+mz+p+\gamma=-e^{h\gamma}\int_{\R}k(y)e^{-z y}dy\\ \label{t2}
 2z+m=e^{h\gamma}\int_{\R}yk(y)e^{-zy}dy 
 \end{eqnarray} 
 The usual assumptions on  $k$ in the literature  (see, e.g., \cite{HZ,TAT,YCW,WLR}) are that either  
 $$\int_{\R}e^{\lambda y}k(y)dy \quad\hbox{exists for all}\quad \lambda\in\R
 $$ 
 or if $a$ (or $b$) is a real number then $q_2$ is defined for all $\lambda\in(a,b)$ and 
 $$
 \lim_{\lambda\to a^+}\int_{\R}e^{\lambda y}k(y)dy\  (\hbox{or} \ \lim_{\lambda\to b^-}\int_{\R}e^{\lambda y}k(y)dy)=\infty, 
 $$ 
 and, in both cases, the convex function $E_m(z)$ tend to $+\infty$ at $z=a$ and $z=b$ so that in these cases the assumption {\bf (T)} is satisfied.} 

\bigskip
\noindent
{Moreover, in the context of semi-wavefronts, if $m_{\pm}=c^{\pm}_*$  then the minimality of $m_{\pm}$ for the existence of crossings between $q_1$ and $q_2$ \cite[Theorem 18]{GPT} 
implies the condition {\bf (T)} with $\gamma_m=0$ and $z_m=\lambda_*^{\pm}$. So that,  for critical semi-wavefronts  {\bf (K)} implies {\bf (T)}.}

\bigskip 
 
 \noindent Finally, we introduce the function 
$k_{z}(x)=k(x)e^{-zx}$, $z\in\R$,  the number
\begin{eqnarray}\label{cb}
k^*_{z_m}:=\int_{\R}y^2k_{z_m}(y)dy.
\end{eqnarray}
 and  the  Fourier  transform of $u$ as
 $$
 \hat{u}(z)=\int_{\R}e^{-izx}u(x)dx.
 $$

\bigskip
\noindent
With all this notation we state our main result. 
\begin{thm}\label{sba}
{Assume  {\bf(K)} and {\bf (T)}}. If the initial data $u_0(s,\cdot)=u_0(\cdot), s\in[-h,0]$, to (\ref{dle0}) is such that 
$$
\int_{\R}e^{-z_my}|u_0(y)|dy<\infty,
$$ 
then for each $m\in\mathcal{M}$ 
\begin{eqnarray}\label{asym}
\lim_{t\rightarrow\infty}\sqrt{t} \, e^{\gamma_m t}u(t,{a(t,x))=\frac{1}{2\sqrt{\pi\sigma_m}}}\int_{\R}u_0(y)dy \ e^{z_m x} \quad \hbox{for all} \ x\in\R,
\end{eqnarray}
 where
 \begin{eqnarray}\label{mq}
{a(t,x)=o(\sqrt{t}) \quad\hbox{and}\quad \sigma_m= \frac{2+k_{z_m}^*e^{\gamma_m h}}{2(1+he^{\gamma_mh}\hat{k}_{z_m}(0))}.}
 \end{eqnarray} 
 {Moreover, without assuming {\bf (T)}, if for some $z_0\in\R$ the initial data $u_0$ satisfies
$$
\int_{\R}e^{-z_0 y}|u_0(y)|dy<\infty,
$$ 
 then there exist $C>0$ such that 
 \begin{eqnarray}\label{O}
 |u(t,x)|\leq Ct^{-1/2}e^{-\gamma_0 t}\ e^{z_0 x} \sup_{s\in[-h,0]}||e^{-z_0 (\cdot)}u_0(s, \cdot)||_{L^1(\R)}\quad \hbox{for all} \quad (t,x)\in\R_+\times\R
 \end{eqnarray} 
where $\gamma_0$ is defined as the only real solution to
\begin{eqnarray}
-\gamma_0+q_1(z_0)=\hat{k}_{z_0}(0)e^{\gamma_0 h}.\end{eqnarray}}
\label{gam}
\end{thm}

\bigskip
\bigskip

Therefore, for equations of type 
 (\ref{dle0}) associated with the stability of semi--wavefronts, i.e., for the equation
 \begin{eqnarray}\label{dle2}
u_t(t,x)=u_{xx}(t,x)-mu_x(t,x)-u(t,x)+ [{k_0} * u(t-h,\cdot-mh)](x), 
\end{eqnarray} 
where $(t,x)\in\R_+\times\R$, we have the following {result on the optimality of the convergence rates where the speeds set $\mathcal{M}_0$ is determined by  $q_1(z)=-z^2+mz+1$ and $q_2=g'(0)e^{-zmh}\int_{\R}k_0(y)e^{-zy}dy$ as the numbers $m\in\mathcal{M}_0$ such that the curves $q_1$ and $q_2$ have at least one intersection.}

 \begin{coro}\label{coro1}{Assume {\bf (K)} and {\bf (T)}}. Let $m\in\mathcal{M}_0$ and the kernel $k_0$ be such that
$$
k^*_m:=\int_{\R}y^2k_0(y-mh){e^{-z_m y}}dy<\infty.
$$
If the initial data  $u_0$ for (\ref{dle2}) is such that  $e^{-{z_m} \cdot}u_0(\cdot)\in C([-h,0],L^1(\R))$ then, 
\begin{eqnarray}
\lim_{t\rightarrow\infty}\sqrt{t} \, e^{\gamma_m t}u(t,x+{a(t,x)})=e^{z_m x}\frac{1}{2\sqrt{\pi\sigma_m}}\int_{\R}u_0(y)dy\quad \hbox{for all}\ x\in\R,
\end{eqnarray}
where  ${a(t,x)=o(\sqrt{t})}$ and $\sigma_m$ is given by (\ref{mq}).  
\end{coro}

\bigskip
\noindent
{It is instructive compare our method with the approach used by Huang {\it et al} in \cite[Section 3]{HMW} where the authors consider the stability of monotone  wavefronts in 
non-local dispersive equations on the $n$--dimensional space where $k_0$ is  a heat kernel. In \cite{HMW}, the estimate for the Fourier transform, in the $L^2$-sense, of solutions is 
obtained by using so-called {\it delayed exponential function} (see \cite{KIK}) and   the  optimality of convergence rates  is stated  with respect to the estimate function obtained in \cite[Proposition 1]{HMW},
 i.e.,  the convergence rate in \cite{HMW} for equation (\ref{dle0}) is $O(t^{-n/2}e^{-\gamma' t})$, some $\gamma'\geq 0$,  and $q_1(z_m)=q_2(z_m)$ implies $\gamma'=0$. 
 Our approach, which is developed for the one-dimensional space, shows that  in a suitable space, the rate of convergence is  actually $o(t^{-1/2}e^{-\gamma_m t})$ and $\gamma_m=0$ when $q_1(z_m)=q_2(z_m)$. 
 Moreover, for the $n$--dimensional case we can also give the estimation $O(t^{-n/2}e^{-\gamma_m t})$ (see Remark \ref{multi}) and 
 it is interesting to see the $n$-dependence of our universal estimate function (for the optimality of this function see Remark \ref{opth}) because   the $L^1$-norm of 
 this universal function turns out to be asymptotic to Gamma function. More precisely, our approach strongly depends on the integrability, as $t\to+\infty$, of the 
 function  $$E_n^h(r):=\frac{r^{n-1}}{[1+ \frac{r^2}{t/h}]^{t/h}},$$  and clearly $\lim_{t\to+\infty}||E_n^h(t,\cdot)||_{L^1(\R_+)}= \Gamma(n/2)/2$ for all $h>0$.}

\bigskip
\bigskip
\bigskip
\noindent
An important application deriving from the asymptotic behavior of the solutions of  (\ref{dle0}) is that we can prove a lower bound on the measure of the level sets of the solution
$u(t,\cdot)$, for every $t$, of (\ref{nlne}) when the initial data decays sufficiently fast. Known results on the stability of wavefronts (see, e.g., \cite{CMYZ,HMW,LLLM,MOZ,S2})
provide information on the speed of propagation for initial data which are asymptotic to a  semi-wavefront, i.e., for initial data that behave like $z^je^{\lambda z}$,  $j=1, 2$ for some $\lambda\in\R$. 
 However, as far as we know, there are no studies in the case $h>0$  that provide results on the asymptotic behavior of the solutions of  (\ref{nlne}) 
 with an initial data that decays faster than exponential. In this manuscript we will provide information on the level sets of the solutions of 
 (\ref{nlne}) in that case.

\bigskip
\noindent
 In the local case without delay (i.e., when $k_0=\delta$ and $h=0$), after the seminal work of
Kolmogorov $et \ al$ \cite{KPP},  McKean \cite{Mc} using probabilistic methods established an estimate depending on the logarithm of $t$ for the distance 
between the level sets of wavefronts with minimal speed and the solutions of (\ref{nlne})  with a Heaviside initial data, i.e., $$\mathcal{D}(t):=m(t)+ct\leq \frac{1}{2\lambda_*}\log (t)+B,$$ where $m(t)$ is such that $u(t, m(t))=1/2$ (here $u(t,x)$ is the solution generated by the Heaviside initial datum), $\lambda_*:=\lambda_*^+=\lambda^+_*(h)$ and $B\in\R$.  Uchiyama  \cite[Proposition 9.1]{UC} 
proved the same estimate as an immediate consequence of the fundamental solution of Heat equation and by using probabilistic arguments he obtained $\mathcal{D}(t)=\frac{3}{2\lambda_*}\log (t)+O(\log\log(t))$. A precise description, i.e., $\mathcal{D}(t)=\frac{3}{2\lambda_*}\log (t)+O(1)$ was given by  Bramson \cite{Br}  using a probabilistic approach. Then, Ebert and van Saarloos \cite{ES} heuristically given more refined representation for $\mathcal{D}(t)=\frac{3}{2\lambda_*}\log (t)+x_0-3\sqrt{\pi}/\sqrt{t}+O(1/t)$ which was recently confirmed by probabilistic arguments  \cite{NRR} . Noteworthily,   Hamel {\it et al}  \cite{HNRR1} showed $\mathcal{D}(t)=\frac{3}{2\lambda_*}\log (t)+O(1)$ by only using PDE arguments. Furthermore, a recently work of Bouin {\it et al} \cite{BHR},  for a certain  non-local version of  KPP equation the authors show that, depending on $k_0$, the correction term can be $\mathcal{D}(t)=O(t^{\beta})$ for some $\beta\in(0,1)$. Due to the similarity given in (\ref{asym}) between the asymptotic for $h=0$ and $h>0$ we believe that, at least for  local equations, the term $\frac{3}{2\lambda_*}\log(t)$ must persist when $h>0$,  therefore  we propose the following problem

\bigskip

{\bf Open problem }  To prove that if $k_0(x)=\delta(x)$ in (\ref{nlne}) then $\mathcal{D}(t)= \frac{3}{2\lambda^+_*(h)}\log(t)+O(1)$ for 
\indent all $h>0$.

\bigskip
\noindent
However, the main difficulty in the case with delay is how to obtain 
an explicit expression for the  fundamental solution of (\ref{dle0}). In spite of the fact that there are abstract expressions for the fundamental solutions of  (\ref{dle0})
(see, e.g., \cite{Nk} and \cite[Section 4.5]{Wu}) in this case, their derivation seems quite intricate. {In fact,  Nakagiri in \cite[Theorem 4.1]{Nk} gives an expression for the fundamental solution
of (\ref{dle0}) in terms of the fundamental solution of the heat equation but defined piecewise on the succesive intevals $[0,h],[2h, 3h]....$. Following the theory of fundamental solutions in \cite[Chapter 1]{AF} 
we make the following definition 
\begin{definition}\label{def}
A {\it fundamental solution} for (\ref{dle0}) is a function $\Gamma_h(t,x)$ defined for all $x\in\R$ and $t>0$ which satisfies the following conditions 
\begin{itemize}
\item[(i)]  $\Gamma_h(t, x)$ as function of $(t,x)$ satisfies (\ref{dle0}) for each $y \in\R$,
\item[(ii)]   if  $\psi\in C([0,h], L^1(\R))$ then 
$$
\lim_{t\to 0}\int_{\R}\Gamma_h(t,x-y)  \psi(s,y)dy=\psi(s,x)\quad\hbox{for all}\quad (s,x)\in[-h,0]\times\R.
$$ 
\end{itemize}
\end{definition}}
\vspace{3mm}
On the other hand if we use Fourier Transforms, the natural 
fundamental solution to  (\ref{dle0}) is the function
 \begin{eqnarray}\label{sf}
\Gamma_h(t,x):= \int_{\R}e^{ixy}e^{\lambda(y)t}dy,
\end{eqnarray}
where $\lambda:\R\to\C$ is a function satisfying
\begin{eqnarray}\label{lamb}
\lambda(z)=-z^2 +imz+p+\hat{k}(z)e^{-h\lambda(z)}.
\end{eqnarray}

The  immediate problem  is how to define globally a function $\lambda$ implicitly through (\ref{lamb}) which, for each $z \in \R$ has infinitely many solutions. 
Even though for our main result we only need to define $\lambda$ locally in a neighborhood of $0$, we discuss the situation in general, 
which is of independent interest. In the case when 
$\hat{k}(\cdot-mh)$ is positive (this condition is made, e.g., in \cite{BNPR}) then, 
$\lambda(z)=\rho(z)+imz$ where $\rho:\R\to\R$ is the unique solution of 
\begin{eqnarray}\label{rho}
\rho(z)=-z^2+p+\hat{k}(z-mh)e^{-h\rho(z)}.
\end{eqnarray}
The main difficulty of this technique has to do with the integrability of the function  $e^{\lambda(\cdot)t}$. 
For example in the local case, i.e., when formally $\hat k$ is equal to a positive constant $q$, by taking $m=0$ and $p=-q$ in (\ref{dle0}), the function $e^{\lambda(\cdot)t}$ is no longer a Gaussian when $h>0$, since
$e^{\lambda(z)t} \sim (q/z^2)^{\frac{t}{h}}$ for  $z\to\pm\infty$ for all $t>0$ (see Remark \ref{opth}). Thus, even though 
$\Gamma_h$ is formally a fundamental solution to (\ref{dle0}), by the definition of $\lambda$ in (\ref{lamb}), with $h>0$, we cannot define $\Gamma_h(t,\cdot)$ for $t\in[0,h/2]$. This is the main difference between 
the cases $h=0$ and  $h>0$.

\bigskip

In the non local case the situation can be different  since {a suitable convergence  of $\hat{k}$ at $\pm\infty$} can be used to insure the necessary integrability of $e^{\lambda(\cdot)t}$ 
(see {Remark \ref{opth}}). In fact, if $\hat{k}(\cdot+mh)$ is positive and $\hat{k}(z)=O(e^{-hz^2})$ then $$\Gamma_h(t,x)=\int_{\R}e^{i(x+mt)y}e^{[\rho(y)+\gamma]t}dy,$$ with $\rho(y)$ given by (\ref{rho}),  
is actually a fundamental solution to (\ref{dle0}) (see Corollary \ref{loger} below). 

\bigskip
\bigskip
\noindent
The rest of the manuscript is organized as folllows: In Section 2 we prove Theorem \ref{sba} and, in Section 3,
we apply this theorem to estimate the level sets of the evolution of solutions of nonlocal nonlinear reaction--difusion equations with delay.  

\section{Proof of Theorem \ref{sba}}

\noindent
We begin this section by proving a Halanay \cite{Hal966} type result. 

\begin{lem}[Halanay] \label{halanay}
Let $\mu, k \in \C$ and let  $X$ be a complex Banach space. If for all $h>0$, $r\in C([-h,\infty),X)$ is a function satisfying
\begin{eqnarray*}
r_t(t)= \mu r(t)+kr(t-h),
\end{eqnarray*}
almost everywhere, then, 
 \begin{eqnarray}\label{h1}
|r(t)|\leq (\sup_{s\in[-h,0]} |r(s)|) \, e^{\min\{0,-{\tau} h\}} e^{\tau t},
 \end{eqnarray}
 for all $t> -h$. 
 Here,  $\tau$ is the only real root of the equation
 \begin{eqnarray} \label{elam}
 \tau=Re(\mu)+|k|e^{-\tau h}.
\end{eqnarray}
Moreover, we have that
  
 \begin{enumerate}
 \item[(i)] $\tau\leq 0  \qquad \mbox{if and only if} \qquad -Re(\mu)\geq |k| $, and
 
 \item[(ii)] $\tau = 0 \qquad \mbox{if and only if} \qquad -Re(\mu)= |k|$.
 \end{enumerate}
\end{lem}

\noindent
\begin{proof}It is clear that
$$
 \frac{d}{dt}(r(t)e^{-\mu t})=ke^{-\mu t}r(t-h)
$$
almost everywhere, which implies
\begin{eqnarray}
x(t)=k\int_{0}^{t} e^{\mu(t-s)}x(s-h)ds+x(0),
\end{eqnarray}
and from here we have that $|r(t)|$ satisfies the following inequation
\begin{eqnarray}\label{II}
x(t)\leq |k|\int_{0}^{t} e^{Re(\mu)(t-s)}x(s-h)ds+x(0),
\end{eqnarray}
for all $t > 0$. 
Now notice that due to (\ref{elam}), for $A\in\R$ the function $e_A (t)=Ae^{\tau t}$ satisfies the equation 
$\frac{d}{dt}(e_A(t)e^{-Re(\mu)t})=|k|e_A(t-h)e^{-Re(\mu)t}$, 
and thus,  $e_A(t)$ satisfies  (\ref{II}) with equality. Hence, the function
 $\delta(t)=|r(t)|-e_{A}$, with $A= \sup_{s\in[-h,0]}|r(s)|e^{\min\{0,-\tau\} h}$, for  $t\in[0,h]$ satisfies (\ref{II}) and, thus,  $\delta(t)\leq 0$ for all  $t\in[0,h]$. 
 In an analogous way we conclude that $\delta(t)\leq 0$ on the intervals $[h,2h],[2h,3h]...$  This proves  (\ref{h1}).

\bigskip
\noindent
Let us now prove (i). If $-Re(\mu)\geq |k|$ then $\tau\leq|k|(e^{-h\tau}-1)$ which in turn implies that $\tau\leq 0$. On the other hand,
if $\tau\leq 0$ we assume $-Re(\mu)<|k|$. Hence,  $\tau>|k|(e^{-h\tau}-1)$ which is a contradiction. 

\bigskip
\noindent
In order to prove  (ii) notice that the derivative of  $a(\tau):=\tau-Re(\mu)-|k|e^{-h\tau}$ 
is always positive hence,  $a(\tau)$ has at most one zero. 
If $Re(\mu)=-|k|$ then  $\tau=0$ is the unique zero. If $\tau=0$,  from (\ref{elam}) we conclude that $a=b$.
\end{proof}

\bigskip
\bigskip

\noindent
For some $z_0$, let us define the function $l:\R\to\R$ through the following equation
\begin{eqnarray}\label{lamb1}
l(z)=-z^2 +\gamma_0-q_1(z_0)+e^{h\gamma_0}|\hat{k}_{z_0}(z)|e^{-hl(z)}\quad  \mbox{for $z\in\R$}.
\end{eqnarray}
Recalling that $k_{z_0}(x)=e^{-z_0 x}k(x)$ and that $\gamma_0$ is defined implicitly by   
\begin{eqnarray}\label{O1}
q_1(z_0)-\gamma_0=|\hat{k}_{z_0}(0)|e^{\gamma_0 h}
\end{eqnarray}
we can estimate the function $l(z)$ in (\ref{lamb1}) . 

\bigskip
\noindent
For $\epsilon_h(z):=[1+h|\hat{k}_{z_0}(z)|e^{h\gamma_0}]^{-1}$ we define the function 
$$
\alpha_{h}(z):=-\frac{1}{h}\log(1+h\epsilon_h(z) z^2).
$$
Then we have,

\begin{lem}\label{loge}
The function $l(\zeta)$ satisfies the following estimate, 
\begin{eqnarray}\label{gauss}
-\epsilon_h(z)\ z^2+e^{\gamma_0 h}[|\hat{k}_{z_0}(z)|-\hat{k}_{z_0}(0)]\ \leq \ l(z) \ \leq \ \alpha_{h}(z),
\end{eqnarray}
for all  $z \in \R$ and 
\begin{eqnarray}\label{abl}
\lim_{z\to \pm \infty}  z^2 \, e^{h \, l(z)} = 0.
\end{eqnarray}
Moreover,  if 
\begin{eqnarray}\label{gauss1}
|\hat{k}(z)|e^{hz^2}\leq C,
\end{eqnarray}
for some $C>0$ then
\begin{eqnarray}\label{gauss2}
l(z)/z^2+1= O(z^{-2}).
\end{eqnarray}
\end{lem}
\begin{rem}\label{opth}
Note that for the local equation (i.e., when $\hat{k}=q$, for some positive constant $q$) if we put $m=0$ and $p=-q$ then  the equation (\ref{dle0}) is reduced to 
$$
u_t(t,z)=u_{xx}(t,x)-qu(t,x)+qu(t-h,x),
$$ 
and the upper bound in (\ref{gauss}) is sharp when $h=0$, i.e., $\alpha_h(z)\sim -z^2$ as $h\to 0$. 
However, when $h>0$ the asymptotic behavior for $\alpha_h$ is different since applying the definitions of $\gamma_0$ and $z_0$ in (\ref{O1}) 
we have $\gamma_0=z_0=0$, then  multiplying  (\ref{lamb1}) by $e^{hl(z)}$  we get  
\begin{eqnarray}\label{gq}
e^{hl(z)}l(s)=-z^2e^{hl(z)}-qe^{hl(z)}+q,
\end{eqnarray}
while the upper estimate in (\ref{gauss}) implies  $\lim_{|z|\to\infty}l(z)=-\infty$. So that, from (\ref{gq}), for each $t>0$ we obtain $e^{l(z)t}\sim (q/z^2)^{t/h}$ for $z \to\pm\infty$. 
\end{rem}

\bigskip
For the non local case, by (\ref{gauss1}) if $\hat{k}=O(e^{-hz^2})$ then  by (\ref{gauss2}) we have $e^{l(z)t}\in L^1(\R)$ for all $t>0$ and therefore we obtain the following result

\begin{coro}[Fundamental solutions]\label{loger}
If  the kernel $k$ satisfies $\hat{k}(z-mh)\in\R_+$ for all $z\in\R$ and $\hat{k}(z)=O(e^{-hz^2})$ then $\Gamma_h(t,x)$ defined by (\ref{sf}), with $\lambda(z)=\rho(z)+imz$ and $\rho$ given by (\ref{rho}), is a fundamental solution for (\ref{dle0}).
\end{coro}
\begin{proof}
If in (\ref{lamb1}) we take $z_0=0$ we have $-q_1(z_0)=p$ and $$|\hat{k}_{z_0}(\cdot)|=|\hat{k}(\cdot)|=|e^{ihm(\cdot)}\hat{k}(\cdot)|=|\hat{k}(\cdot-mh)|$$ therefore $l=\rho$ and $\gamma_0$ is given by (\ref{O1}), i.e., $-p-\gamma_0=|\hat{k}(0)|e^{\gamma_0 h}$. Thus,  by (\ref{gauss2}) we have $|e^{\lambda(z)t}|= e^{\rho(z)t}\sim e^{-z^2 t}$ at $z=\pm\infty$ for all $t>0$, therefore   $e^{\rho(\cdot)t}\in L^1(\R)$ for all $t>0$. So that  $\Gamma_h$ is well defined for all $t>0$ and  $\Gamma_h(t,\cdot)\in L^{\infty}(\R)$ for each $t>0$. Moreover,  by the definition (\ref{lamb})  of $\lambda$ then  $\Gamma_h$ solves (\ref{dle0}) for all $t>0$, therefore the condition (i) in Definition \ref{def} is immediately satisfied. Otherwise,
 $\Gamma_h(t,\cdot)  u_0(s,x-\cdot)\in L^1(\R)$ for all $(t,x)\in\R_+\times\R$ and evaluating in $x=0$  we get
\begin{eqnarray}\label{1}
   \Gamma_h(t,\cdot) * u_0(s,\cdot)(0)= \int_{\R}\Gamma_0(t,y)u_0(s,-y)dy+\int_{\R}\Theta(t,y)u_0(s,-y)dy,
\end{eqnarray}
where $\Gamma_0(t,y)=\frac{e^{\gamma t}}{2\sqrt{\pi t}}e^{-y^2/4t}$ and $\Theta(t,y)=\Gamma(t,y)-\Gamma_0(t,y)$. Next, by the definition of $\Gamma_h(t,y)$, for all $t>0$ we have
$$
\Theta(t,y)=e^{\gamma t}\int_{\R}e^{i\tau y}[e^{\rho(\tau)t}-e^{-\tau^2t}]d\tau
$$
and using (\ref{gauss2}) there exists $C'>0$ such that $|\tau^2+\rho(\tau)|\leq C'$ for all $\tau\in\R$, therefore  
$$
|\Theta(t,y)|\leq e^{\gamma t}\int_{\R}e^{-\tau^2t}|1-e^{-(\rho(\tau)+\tau^2)t}|d\tau\leq e^{C' t}(1-e^{-C't})e^{\gamma t}\int_{\R}e^{-\tau^2 t}d\tau.
$$
However, 
$$
(1-e^{-C't})\int_{\R}e^{-\tau^2 t}d\tau=\frac{1-e^{C't}}{\sqrt{t}}\sqrt{\pi}\to 0,
$$
as $t\to 0$. Thus, by passing to the  limit $t\to 0$ in (\ref{1}) we obtain the condition (ii) in Definition \ref{def} for $x=0$. 
Finally, for each $x'\in\R\setminus \{0\}$ take the initial datum $w_0(s,x):=u_0(s,x+x')$ and note that $w(t, x)=u(t,x+x')$ for all $t>-h$ and $x\in\R$ which completes the proof.
\end{proof}

\bigskip

\begin{proof}[Proof of Lemma \ref{loge}]
Let us denote $Q(z)=|\hat{k}_{z_0}(z)e^{\gamma_0h}|$, $P=\gamma_0-q_1(z_0)$,and $Q_0=|\hat{k}_{z_0}(0)|e^{h\gamma_0}$. If we define  $\beta(z):=-\alpha_h(z)+l(z)$ then 
$$
 \beta(z)=\frac{1}{h}\log(1+h\epsilon_h(z)z^2)-z^2+P+Q(z)(1+h\epsilon_h(z)z^2)e^{-h\beta(z)}.
$$
 From Lemma \ref{halanay}  we have that $\beta(z)\leq 0$ if and only if:
 \begin{eqnarray}\label{desl}
z^{2}-\frac{1}{h}\log(1+h\epsilon_h(z)z^2)-P\geq Q(z)(1+h\epsilon_h(z)z^2).
\end{eqnarray}
Now, using $\log(1+x)\leq x$, for all $x\geq 0$, in order to obtain (\ref{desl}) it is enough to have
\begin{eqnarray*}
z^2-\epsilon_{h}(z)z^2-P\geq Q_0+hQ(z)\epsilon_h(z)z^2\quad\hbox{for all}\quad z\in\R\\
\iff [1-\epsilon_{h}(z)-Q(z)h\epsilon_h(z)] z^2-P- Q_0\geq 0\quad\hbox{for all}\quad z\in\R\\
\iff  -P\geq Q_0.
\end{eqnarray*} 

\noindent
 Thus, as $P=-Q_0$ then $\beta(z)\leq 0$ for all $z\in\R$ which implies  the right hand of (\ref{gauss}).  
 Otherwise, using (\ref{lamb1}) and the right hand of (\ref{gauss})
\begin{eqnarray*}
l(z)&\geq&-z^2-Q_0+Q(z)(1+h\epsilon_h(z)z^2)\\
&=& -\epsilon_h(z)z^2+Q(z)-Q_0
\end{eqnarray*}
This proves the left hand of (\ref{gauss}).

\noindent
Finally, note that if $r_j$ is the only real solution for
\begin{eqnarray}\label{r}
r_j=P+Q_je^{-hr_j}\quad\quad j=1,2
\end{eqnarray}
with $0\leq Q_1\leq Q_2$ then $r_1\leq r_2$.  Therefore, as $r(z)=l(z)+z^2$ satisfies (\ref{r}) with $Q_j=Q(z)e^{hz^2}$ 
we conclude that $P\leq r(z)\leq r_C$ for all $z\in\R$  where $r_C$ is the solution of (\ref{r}) with $Q_j=C$, which implies (\ref{gauss2}).
 \end{proof}

\begin{rem}\label{zeta}
Similarly, by taking $\beta(z)=l(z)-\alpha_h(z)$ but this time with $l,\alpha_h:\R^n\to\R$ defined by  
$$
l(z)=-|z|^2+P+Q(z)e^{-hl(z)},
$$ 
and 
$$
\alpha_h(z):=-\frac{1}{h}\log(1+h\epsilon_h(z)|z|^2)
$$
it is an exercise completely analogous to show  that  the upper bound in (\ref{gauss})  holds. 
So that
\begin{eqnarray}
l(z)\leq \alpha^*_h(z):=-\frac{1}{h}\log(1+h\epsilon_h|z|^2),
\end{eqnarray}
where the number $\epsilon_h:=1/[1+hQ_0e^{-\gamma_0 h}]$.
\end{rem}

\bigskip
\bigskip

\noindent
Let us consider the following characteristic equation, 
\begin{eqnarray}\label{ace}
L(s)=-s^2+i(2z_m+m)s-q_1(z_m)+\hat{k}_{z_m}(s)e^{-hL(s)}.
\end{eqnarray}
Since the linear transformation $D:\C\to\C$, defined by $D(s)=(1+h\hat{k}_{z_m}(0)e^{\gamma_m h})s$, is invertible then there is  a unique analytic function  
$L:B_{\delta_0}(0)\subset\C\to\C$, with  $\delta_0>0$, such that $L(0)\in\R$. As $L(0)$ satisfies (\ref{t1}) then  $L(0)=-\gamma_m$.
 
 \bigskip
 \noindent
 Then we have the following result.

 \begin{lem} \label{liml} If we assume {\bf(K)} is satisfied, then we have
 \begin{eqnarray}
 \lim_{s\to 0}(L(s)+\gamma_m)/s^2=-\frac{k_{z_m}^*e^{\gamma_m h}+2}{2(1+he^{\gamma_mh}\hat{k}_{z_m}(0))}=-\sigma_m.
 \label{sigmam}
  \end{eqnarray}
 \end{lem}
 
 \begin{proof} It follows from (\ref{ace}), the tangency of the curves $q_1-\gamma_m$ y $q_2e^{\gamma_m h}$ (by an appropriate choice of $\gamma_m$ as we discussed above)
 and the hypothesis {\bf(K)} that
 \begin{eqnarray*}\label{lim}
\lim_{s\rightarrow 0}\frac{L(s)+\gamma_m}{s^2}&=&\frac{1}{2}\lim_{s\rightarrow 0}\frac{L'(s)}{s}\\
&=&\frac{1}{1+he^{\gamma_mh}\hat{k}_{z_m}(0)}\left[-1+\frac{1}{2}\lim_{s\to0}\frac{(2z_m+m)i+\hat{k}'_{z_m}(s)e^{-hL(s)}}{s}\right]. \\
\end{eqnarray*}
However,  $k'_{z_m}(s)=-i\int_{\R}ye^{-isy}k_{z_m}(y)dy$ therefore (\ref{t2}) implies $k'_{z_m}(0)=-i(2z_m+m)$ so that
\begin{eqnarray*}
\lim_{s\rightarrow 0}\frac{L(s)+\gamma_m}{s^2}&=&\frac{1}{1+he^{\gamma_mh}\hat{k}_{z_m}(0)}\left[-1+\frac{1}{2}\lim_{s\to0}(\hat{k}''_{z_m}(s)-hL'(s)\hat{k}'_{z_m}(s))e^{-hL(s)}\right]\\
&=&-\frac{k_{z_m}^*e^{\gamma_m h}+2}{2(1+he^{\gamma_mh}\hat{k}_{z_m}(0))}.
\end{eqnarray*}
\end{proof}
\bigskip
\noindent
With all the previous results we are ready to prove the main result of this section. 

\bigskip

\bigskip
\begin{proof} [Proof of Theorem \ref{sba}]
Making the change of variable  $v(t,x)=e^{-z_mx}u(t,x)$  in (\ref{dle0}) and $v_0(s,x)=e^{-z_mx}u_0(s,x)$, and using, as before, that $k_{z_m}(x) = k(x) \exp(-z_m x)$, we see that $v$ satisfies, 
\begin{eqnarray}\label{dle1}
v_t(t,x)=v_{xx}(t,x) +(2z_m+m)v_x(t,x) 
-q_1(z_m)v(t,x)+  k_{z_m}(\cdot) * v(t-h, \cdot) (x).
\end{eqnarray}
Making the further change of variables $v \to \alpha$ given by 
$\alpha(t,x):=\exp{(q_1(z_m)t)}v(t,x-(2z_m+m)t)$ we see from (\ref{dle1}) that $\alpha(t,x)$ satisfies, 
\begin{eqnarray}\label{heat}
\alpha_t(t.x)=\alpha_{xx}(t,x)+f(t,x)
\end{eqnarray}
for $(t,x)\in[0,h]\times\R$. 
In (\ref{heat}), the second term of the right, given by 
$$
f(t,x):=\int_{\R}e^{-q_1(z_m)h}k_{z_m}(x-y)\alpha(t-h,y-(2z_m+m)h)dy,
$$ 
is Lebesgue  integrable for all $t\in[0,h]$. 
Since the fundamental solution of (\ref{heat}), given by $\Lambda(t,x)=\frac{1}{2\sqrt{\pi t}}e^{\frac{-x^2}{4t}}$,  satisfies 
$\Lambda(t,\cdot)\in C^{\infty}(\R)\cap W^{2,1}(\R)$ for all  $t>0$, using Duhamel's formula we conclude that 
$v(t,\cdot)\in C^{2}(\R)\cap L^1(\R)$ for all $t\in[0,h]$. Iterating this procedure successively on the intervals
$[h,2h], [2h,3h], \dots$ we conclude that  $v(t,\cdot)\in C^{2}(\R)\cap L^1(\R)$ for all $t>0$.
For convenience we introduce, 
$$
A=\int_{\R}v_0(y)dy.
$$
Given the smoothness of $v(t,x)$ as a function of $x$ for $t>0$, the Fourier Inversion Theorem implies that
\begin{eqnarray} \label{fou}
\sqrt{t} \, v(t,x)={\frac{1}{2\pi}}\int_{\R}e^{ixy/\sqrt{t}}\hat{v}(t,y/\sqrt{t})dy.
\end{eqnarray}
for $t>0$ and $x \in \R$. Moreover the function $e^{\gamma_m t}\hat{v}(t,z)$ satisfies the equation
\begin{eqnarray}\label{cplx}
w_t(t,z)=(-z^2+i(2z_m+m)z+\gamma_m-q_1(z_m))\,w(t,z) & \nonumber \\
&+\hat{k}_{z_m}(z)e^{\gamma_m h} \, w(t-h,z)
\end{eqnarray}
for $t>0$ and all $z\in\R$.

\bigskip
\noindent
{Next, we define $l:\R\to\R$, implicitly, by the equation 
\begin{eqnarray}
 l(z)=-z^2+\gamma_m-q_1(z_m)+|\hat{k}_{z_m}(z)|e^{\gamma_m h}e^{-l(z)h}.
\end{eqnarray}}

\bigskip
\noindent
Note that the pair $(\gamma_m, z_m)$ satisfies (\ref{O1}) and therefore by Lemma \ref{loge} we have that  $l(z)\leq 0$ for all $z\in\R$. 
Hence, by applying  Lemma \ref{halanay} to (\ref{cplx}) (with $\tau=l(z)$), we have,
\begin{eqnarray}
|e^{\gamma_m t}\hat{v}(t,y/\sqrt{t})|\leq\sup_{s\in[-h,0]}|e^{\gamma_m s}\hat{v}(s,y/\sqrt{t})|e^{l(y/\sqrt{t})t}.
\end{eqnarray}
Using  (\ref{gauss}), and the fact that  $(1+x)^r\geq 1+rx$, for all $-1 \le x <\infty$ and $1 \le r < \infty$ we have, 
\begin{eqnarray}\label{berno}
|e^{\gamma_m t}\hat{v}(t,y/\sqrt{t})|&\leq&\sup_{s\in[-h,0]}||e^{\gamma_m s}v(s,\cdot)||_{L^1(\R)}\frac{1}{[1+h\epsilon_h(y/\sqrt{t}) \frac{y^2}{t}]^{\frac{t}{h}}}\\
 &\leq& \sup_{s\in[-h,0]}||e^{\gamma_m s}v(s,\cdot)||_{L^1(\R)}\frac{1}{1+\epsilon_h(y/\sqrt{t}) y^2},\label{li} 
 \end{eqnarray}
 for all $y\in\R$ and all $t>h$. 
 
 \bigskip
 \noindent
Here, note that  the last inequality (\ref{li}) was obtained without assuming {\bf (T)} for the pair $(\gamma_m,z_m)$. In fact  we only use the fact that this pair satisfies (\ref{O1}) and Lemma \ref{halanay}. 
Therefore (\ref{li}) also holds for  the pair $(\gamma_0, z_0)$ defined in (\ref{O}), so that (\ref{li}) and (\ref{fou}) imply  (\ref{O}).
 
 \bigskip
 \noindent
Now, we use {\bf (T)} in order to obtain the asymptotic behavior of $v(t,\cdot)$. By the Riemann-Lebesgue theorem, there exists $M>0$ such that  $\epsilon_h(z)>M$ 
for all $z\in\R$, so that (\ref{li}) implies $e^{\gamma_m t}|\hat{v}(t,y/\sqrt{t})|$ is dominated by an integrable function.
Thus, to compute $\lim_{t\rightarrow\infty}e^{\gamma_mt}\hat{v}(t,y/\sqrt{t})$ for all  $|y|<\delta_0\sqrt{t}$ we write,
\begin{eqnarray}\label{trian}
|e^{\gamma_m t}\hat{v}(t,y/\sqrt{t})-Ae^{-\sigma_m y^2}|\leq 
 |A \, e^{[L(y/\sqrt{t}){+}\gamma_m]t}-A \, e^{-\sigma_m y^2}|+e^{\gamma_m t}|Ae^{L(y/\sqrt{t})t}-\hat{v}(t,y/\sqrt{t})| .
\end{eqnarray}
Now set, 
\begin{eqnarray*}
I_1(t)&=&|Ae^{[L(y/\sqrt{t})+\gamma_m]t}-Ae^{-\sigma_m y^2}|, \qquad \mbox{and}\\
I_2(t)&=&|Ae^{[L(y/\sqrt{t})+\gamma_m]t}-e^{\gamma_m t}\hat{v}(t,y/\sqrt{t})|.
\end{eqnarray*}
Then, because of (\ref{sigmam}) we have
$$ 
 \lim_{t\rightarrow\infty}I_1(t)=0 \quad\hbox{for all} \ y\in\R
 $$
On the other hand, due to the definition of $L$ in (\ref{ace}),   $e^{[L(z)+\gamma_m]t}$ satisfies (\ref{cplx}) for 
$z\in(-\delta_0,\delta_0)$ and for all  $t\in\R$. Next, by applying Lemma \ref{halanay} to (\ref{cplx}) with $r(t)= e^{\gamma_m t}\hat{v}(t, y/\sqrt{t})-Ae^{(\gamma_m+L(y/\sqrt{y}))t}$ 
and $\tau=l$,  for $|y|<\delta_0\sqrt{t}$ and $t>h$ 
we get
\begin{eqnarray*}
I_2(t)=|e^{\gamma_m t}\hat{v}(t,y/\sqrt{t})-Ae^{\gamma_m t}e^{ L(y/\sqrt{t})t}|&\leq&\sup_{s\in[-h,0]}|e^{\gamma_m s}\hat{v}(s,y/\sqrt{t})-Ae^{L(y/\sqrt{t})s}|e^{l(y/\sqrt{t})t}\\
&\leq&\sup_{s\in[-h,0]}|e^{\gamma_m s}\hat{v}(s,y/\sqrt{t})-Ae^{L(y/\sqrt{t})s}|\frac{1}{1+\epsilon_h(y) y^2}.
\end{eqnarray*}
However, for each $s\in[-h,0]$, $\lim_{t\to\infty}\hat{v}(s,y/\sqrt{t})=A$ and $\lim_{t\to\infty}L(y/\sqrt{t})=\gamma_m$, so that
$$
\lim_{t\rightarrow\infty}I_2(t)=0\quad\hbox{for all} \ y\in\R.
$$
Finally (\ref{asym}) follows from  (\ref{berno}), (\ref{trian}), the Dominated Convergence Theorem and  (\ref{fou}) replacing $x$ by $a(t,x)$.

\end{proof}

\bigskip
\noindent
{\begin{rem}\label{multi}[multidimensional case]
Note that without assuming  the tangential condition {\bf (T)} for the equation (\ref{dle0}) we can simply define the pair  $(\gamma_0,z_0)$ by (\ref{O}) with the initial datum $u_0(s,x)$ satisfying $e^{z_0\cdot x}u_0(s,x)\in C([-h,0],L^1(\R^n))$, for $z_0\in\R^n$ and $n\in\Z_+$, and then by the same arguments of the Proof of the Theorem \ref{sba} we can obtain estimations for equation (\ref{dle0}) when $x\in\R^n$. Indeed, similarly to (\ref{dle1}) the function $v(t,x)=e^{-z_0\cdot x}u(t,x)$ satisfies the equation 
\begin{eqnarray}\label{dlem}
v_t(t,x)  =\Delta v(t,x)+ (2z_0+m)\cdot\nabla v(t,x)-q_1(z_0)v(t,x) \nonumber +\int_{\mathbb{R}^n} k_{z_0}(x-y) \, v (t-h, y)\, dy, \
\end{eqnarray}
for $x\in\R^n, \ t>0$; here the parameter $m\in\R^n$ and  $q_1(z_0)=|z_0|^2+m\cdot z+p$ for $p\in\R$. In this  case, $e^{\gamma_0 t}\hat{v}(t,z)$ satisfies the following equation
\begin{eqnarray}\label{cplx1}
w_t(t,z)=(-|z|^2+i(2z_0+m)\cdot z+\gamma_0-q_1(z_0))\,w(t,z) +\hat{k}_{z_0}(z)e^{\gamma_0 h} \, w(t-h,z).
\end{eqnarray}
Next, by applying Lemma \ref{halanay} to (\ref{cplx1}) and Remark \ref{zeta} ]
\begin{eqnarray*}
\int_{\R^n}|e^{\gamma_0 t}\hat{v}(t,y)|dy&\leq&\sup_{s\in[-h,0]}||v_0(s,\cdot)||_{L^1(\R)}\int_{\R^n}\frac{dy}{[1+h\epsilon_h |y|^2]^{\frac{t}{h}}}\\
 &=& \sup_{s\in[-h,0]}||v_0(s,\cdot)||_{L^1(\R)}\int_{0}^{+\infty}\int_{\partial B(0, r)}\frac{dS}{[1+h\epsilon_h r^2]^{t/h}}dr\\
&=& \frac{\sup_{s\in[-h,0]}||v_0(s,\cdot)||_{L^1(\R)}\ \pi^{n/2}}{\Gamma(n/2) \ t^{n/2}}\int_{0}^{+\infty}\frac{s^{\frac{n}{2}-1}}{[1+\epsilon_h\ \frac{s}{t/h}]^{t/h}}ds.\\
 \end{eqnarray*}
 However, notice that
 $$
 \int_{0}^{+\infty}\frac{s^{\frac{n}{2}-1}}{[1+\epsilon_h\ \frac{s}{t/h}]^{t/h}}ds\leq \frac{1}{\epsilon_h^{n/2}}\int_0^{+\infty}s^{\frac{n}{2}-1}e^{-r}dr=\frac{1}{\epsilon_h^{n/2}}\Gamma(n/2),
 $$
  so that by Fourier's inversion formula
 \begin{eqnarray*}
 e^{-z_0\cdot x}|u(t,x)|\leq\frac{1}{(2\pi)^n}\int_{\R^n}\hat{v}(t,y)dy\leq \frac{ e^{-\gamma_0 t}}{2^{n}(\pi\epsilon_h t)^{n/2}}\sup_{s\in[-h,0]}||v_0(s,\cdot)||_{L^1(\R)}.\\
 \end{eqnarray*}
Note that our estimation $u(t,x)=e^{z_0\cdot x}\, O(e^{-\gamma_0 t}t^{-n/2})$ require minimal conditions on the initial data $u_0$.   In particular, (\ref{O}) is obtained with $n=1$. 
\end{rem}}

\bigskip
\bigskip

\section{Estimation of level set for non-local KPP equations}

\noindent
In this section we study the level sets of the functions
$u:[-h,\infty)\times\R\rightarrow\R$ which satisfy
\begin{eqnarray}\label{nli}
u_{t}(t,x) = u_{xx}(t,x) - u(t,x) + \int_{\R} \, {k_0}(y)g(u(t-h,x-y))dy,
\end{eqnarray}
for all  $t>0, x\in\R$, and
\begin{eqnarray}
u(s,x)=u_0(s,x)\quad (s,x)\in[-h,0]\times\R.
\end{eqnarray}
Here $u_0\in C([-h,0],L^1(\R))$ and  $g$ satisfies

\bigskip
\bigskip
\noindent  {\rm \bf(M)}  
The function $g:\R_+\rightarrow\R_+ $ is such that the equation $g(x)= x$ has exactly two solutions: $0$ and
$\kappa>0$, and  $g(u)\leq g'(0)\,u$ for all  $u\geq 0$. Moreover, $g$  is $C^1$-smooth in some
$\delta_0$-neighborhood of the equilibria  where $g'(0) >1>g'(\kappa)$. In addition,
there are $C >0,\ \theta \in (0,1],$ such that   
$\left|g'(u)- g'(0)\right| +|g'(\kappa) - g'(\kappa-u)| \leq Cu^\theta $ for $u\in
(0,\delta_0].$

\bigskip

\noindent
The condition $g(u)\leq g'(0)u$, for $u\in\R_+$ (i.e., the KPP condition) in {\rm\bf(M)} is satisfied in several models. For example it holds in the Nicholson model where one has
 $g(u)=p \, u \, e^{-au}$ (with $a,p>0$),  or in the  Mackey--Glass model  where $g(u)=pu/[1+au^q]$ (with $a,p>0$ and $q>1$) (see, e.g. \cite{MOZ}). 
 
\bigskip
\bigskip
\noindent
In order to continue with our discussion we need to introduce the following definition.

\begin{definition} 
For $\beta>0$ and a given initial data $u_0$ to (\ref{nli}) we define the function
$m_{\beta}^-(t;u_0):\R_+\rightarrow\R$ as $m_{\beta}^-(t;u_0):=\inf\{x\in\R: u(t,x)=\beta\}$.
Here, $u(t,\cdot)$ is the solution of (\ref{nli}) with the initial data $u_0$ which attains the level $\beta$ at some point on its domain. 
In case the level $\beta$ is not attained, we set  $m_{\beta}^-(t,u_0)=0$. Analogously we define
$m_{\beta}^+(t;u_0):=\sup\{x\in\R: u(t,x)=\beta\}$.
 \end{definition}
 
 \bigskip
 \noindent
 In the context of population dynamics, the functions $m_{\beta}^{\pm}(t,u_0)$ encode the information on the advance of the invading species, with
 initial  population density $u_0$, over a  resident species. In general, the behavior of 
 $m_{\beta}^{\pm}(t,u_0)$ for (\ref{nli}) is unknown.  However, many results have been obtained 
 for (\ref{nli}) in the local case, i.e., for the equation, 
  \begin{eqnarray}\label{nli1}
 w_t(t,x)=w_{xx}(t,x)-w(t,x)+g(w(t-h,x)), 
 \end{eqnarray} 
 for $x\in\R$, and $t>0$ (see, e.g.,  \cite{Br,HNRR1,HNRR2,KPP,Mc,ES}).
 
 \bigskip
 \noindent
 In this respect, the first result on the behavior of $m_{\beta}^{\pm}(t,u_0)$ was obtained in  the classical work of Kolmogorov $et \ al$ \cite{KPP}. They considered
  (\ref{nli1}) with $h=0$ and $g(w)= w-w^2$ and proved that if $u_0$ is a Heaviside function then, 
 $$
 \frac{d}{dt}\left[ \, m_{\frac{1}{2}}^{-}(t;u_0) \, \right] \rightarrow -c_*=-2.
 $$
 Here  $c_*$ denotes the minimal speed for which there exist monotone wavefronts.  
 The actual study of the distance between $m_{\beta}^-(t;u_0)$ and  $-c_*t$ was initiated much later by H.~McKean  who proved a lower bound on 
 $m_{\beta}^{-}(t)+c_*t$ \cite{Mc} using probabilistic methods.  Later Uchiyama  \cite[Theorem 9.1]{UC} was able to obtain McKean's result using the Maximum Principle.

\bigskip
\noindent
Still in the local case, for  $h>0$ and $g$ increasing satisfying {\rm\bf(M)}, it has been proven in   \cite[Theorem 2]{STR1}
that for for $c\geq c_*$ provided  $u_0(x)\sim\phi_c(x)$ (i.e., $\lim_{x\rightarrow-\infty}u_0(x)/\phi_c(x)=1$), where $\phi_c$ is a wavefront,  then one has
 $[w(t,\cdot-ct)-\phi_c]/\phi_c\rightarrow 0$. That is, for all  $\epsilon>0$ there exists  $T_{\epsilon}$ such that
\begin{eqnarray}\label{inq1}
(1-\epsilon)\phi_c(x)\leq w(t,x-ct)\leq (1+\epsilon)\phi_c(x) \qquad  \hbox{for all $(t,x)\in[T_{\epsilon},\infty)\times\R$.}
\end{eqnarray}
{Evaluating (\ref{inq1}) at $x=ct+m_{\beta}^-(t,u_0)$ we conclude that $m^-_{\beta}(t;u_0)+ct$ is bounded for  $\beta\in(0,\kappa)$ since, by taking $\epsilon>0$ such that $(1-\epsilon)\kappa>\beta$,  if there exists a
sequence $\{t_n\}$ such that $m^-_{\beta}(t_n;u_0)+ct_n\to-\infty$ then evaluating  (\ref{inq1}) at $x=m^-(t_n; u_0)$ we have $\beta\leq 0$, which is a contradiction. 
Similarly if  there exists a sequence $\{t_n\}$ such that $m_{\beta}^-(t_n; u_0)+ct_n\to+\infty$ in (\ref{inq1}) we have $(1-\epsilon)\kappa\leq \beta$ a contradiction.}

\bigskip
\noindent
{Recall that under the KPP hypothesis {\bf (M)},  the asymptotic behavior of  a wavefront $\phi_c$ with $c>c_*$ is $\phi_{c_*}(z+z')\sim e^{\lambda_1(c) z}$ where $\lambda_1(c)>0$ is the smallest solution of the characteristic equation $\lambda^2-c\lambda-1+g'(0)e^{-\lambda ch}=0$ and $z'\in\R$; while if $c=c_*$ then the asymptotic behavior for the critical wavefront $\phi_{c_*}$ is  $\phi_{c_*}(z+z')\sim -ze^{\lambda^*z}$ where $\lambda^*=\lambda_1(c_*)$ and $z'\in\R$.} 
Thus, in this case, it only remains to establish what happens with  $m^-_{\beta}(t,u_0)+ct$ for initial data decaying faster than
 $-ze^{\lambda^* z}$ for $z\to-\infty$. The next proposition sheds some light on this issue without assuming monotonicity of $g$.

\begin{pro}\label{nac}
Let  $g$ be  Lipschitz satisfying  {\rm\bf(M)}. 
Consider the level sets for the solutions of  (\ref{nli1}) {for a non negative initial data $u_0\in C([-h,0], L^{\infty}(\R))$ locally Holder continuous in $x\in\R$, 
uniformly with respect to $s\in[-h,0]$, such that $\inf_{s\in[-h,0]}\liminf_{x\to+\infty}u_0(s,x)>0$}.  If $c\geq c_*$ we have that
\begin{itemize}
\item[(a)] There exists a level $\beta_0(g)>0$ such that if  $u_0(s,x)\sim\phi_c(x)$, uniformly on $s\in[-h,0]$, then  $m_{\beta_0}^-(t,u_0)+ct$ is bounded for all $\beta\in(0,\beta_0]$.
\item[(b)] If  $u_0(s,x)\leq Ae^{\lambda_*x}$, for some  $A>0$ and for all  $(s,x)\in[-h,0]\times\R$, then $m_{\beta}^-(t,u_0)+c_*t$ is not bounded for all $\beta>0$.
\end{itemize}
\end{pro}

\begin{proof}

\noindent
 (a) Let us introduce the monotone function $\bar{g}(w)=\max_{x\in[0,w]}g(x)$. It is simple to check that $\bar{g}$ satisfies {\rm\bf(M)} 
 with a positive equilibrium $\bar{\kappa}:=\max_{w \in[0,\kappa]}g(w)$ and  $L_{\bar{g}}:=\sup_{u,v\in\R_+, u\neq v}|\bar{g}(u)-\bar{g}(v)|/|u-v|=g'(0)$. 
 Now, if we denote by $\bar{w}(t,x)$ the solution to (\ref{nli}) with initial data $u_0$ and with  $g=\bar{g}$ it follows from  \cite[Lemma 16]{S} that
\begin{eqnarray}\label{sup}
w(t,x)\leq \bar{w}(t,x) \qquad  \hbox{for all $(t,x) \in[-h,\infty)\times\R$.}
\end{eqnarray}  
Denote by  $\underline{\kappa}=\min_{w\in[\kappa,\bar{\kappa}]}g(w)$. Hence, it follows from the hypothesis {\rm\bf(M)} that one can find an increasing function
$\underline{g}$ satisfying  {\rm\bf(M)} with a fixed point $\underline{\kappa}>0$  and $L_{\underline{g}}=L_{g}=g'(0)$ such that $\underline{g}(w)\leq g(w)$ for all
$w\in[\kappa,\bar{\kappa}]$ (e.g., one can take $\underline{g}$ close to the function $g_0(w):=\min\{w,\underline{\kappa}\}$ in the norm of $C^1(\R_+)$). 
Then, if $\underline{w}(t,x)$ denotes the solution of  (\ref{nli1}) with
initial data $u_0$ and with $g=\underline{g}$ then, using  \cite[Lemma 16]{S}, we have
\begin{eqnarray}\label{inf}
\underline{w}(t,x)\leq w(t,x) \qquad  \hbox{for all $(t,x) \in[-h,\infty)\times\R$.}
\end{eqnarray}
Since $\bar{g}$ and $\underline{g}$ are monotone functions satisfying the KPP condition  then the linear speed $c_*$ is the minimal speed for the existence of wavefronts of (\ref{nli1}) 
for $g=\bar{g}$ and $g=\underline{g}$, respectively. Consequently, for $\bar{g}$ and $\underline{g}$ in (\ref{nli1}) there exist  monotone waves $\bar{\phi}_c$ and $\underline{\phi}_c$, respectively, such that $u_0(s, x)\sim\phi_c(x)\sim \bar{\phi}_c(x)\sim\underline{\phi}_c(x)$ uniformly on $s\in[-h,0]$. Therefore,  by \cite[Theorem 2]{STR1} (without restriction on the strict monotonicity of $\bar{g}$) we conclude 
 $[\bar{w}(t,\cdot-ct)-\bar{\phi}_c]/\bar{\phi}_c\rightarrow 0$ and $[\underline{w}(t,\cdot-ct)-\underline{\phi}_c]/\underline{\phi}_c\rightarrow 0$ for $t\to+\infty$. So that,  if we set $\beta_0:=\kappa/2$, from  (\ref{sup}), (\ref{inf}) we have, 
\begin{eqnarray}\label{sand}
\underline{\phi}_c(x+ct)-\beta/2\leq w(t,x)\leq \bar{\phi}_c(x+ct)+\beta/2 \qquad \hbox{for all $(t,x) \in [T_0,\infty]\times\R$,}
\end{eqnarray}
for some  $T_0=T_0(\beta,u_0)$. Now, notice that if we set $x=m^-_{\beta}(t;u_0)$ in (\ref{sand}) we have that  $m_{\beta}^-(t,u_0)+ct$ must be bounded since if there exists a sequence $\{t_n\}$ such that $m^-_{\beta}(t_n;u_0)+ct_n\to-\infty$ then evaluating  (\ref{sand}) at $x=m^-(t_n; u_0)$ we have $\beta\leq \beta/2$ a contradiction. Similarly if  there exists a sequence  $\{t_n\}$ such that $m_{\beta}^-(t_n; u_0)+ct_n\to+\infty$ in (\ref{sand}) we have $\kappa-\beta/2\leq \beta$ a contradiction.

\bigskip
\noindent
(b)  Assume there exists $C>0$ such that: $|m^-_{\beta}(t,u_0)+c_*t|<C$ for all  $t>0$. Then consider a wavefront  $\bar{\phi}_{c_*}(x)\sim -Axe^{\lambda_* x}$, and using the same notation 
as in (a) consider $b'\in\R$ such that: $(1+\beta)\bar{\phi}_{c_*}(C+b')<\beta$. 
{Next, take $x_0>-1$ such that $u_0(s,x)\leq Ae^{\lambda_* x_0}$ for all $(s,x)\in[h,0]\times\R$ and denote by $\tilde{w}$ the solution to (\ref{nli1}) with this initial data
$$
\tilde{u}(s,x):=\left\{
\begin{array}{lll}
-Axe^{\lambda_*x},& x<- 1\\ 
Ae^{\lambda_*x},& -1\leq x\leq x_0\\
Ae^{\lambda_*  x_0} & x>x_0
\end{array}\right.
$$ 
Clearly $u_0(s,x)\leq \tilde{u}_0(s,x)$ for all $(s,x)\in[-h,0]\times\R$, $\tilde{u}_0(s,x)\sim\bar{\phi}_{c_*}$ uniformly on $s\in[-h,0]$, $\tilde{u}_0(s,\cdot)$ is locally Holder continuous uniformly on $s\in[-h,0]$ and $$\inf_{s\in[-h,0]}\liminf_{x\to+\infty}\tilde{u}_0(s,x)>0.$$ Again, by using  (\ref{sup}) and  \cite[Corollary 1 (inequality (6))]{STR1} we conclude that there exists  
$T_3=T_3(\beta,b')$ such that,
\begin{eqnarray}
w(t,x)\leq \tilde{w} (t,x)\leq(1+\beta)\bar{\phi}_{c_*}(x+c_*t+b') \hbox{for all $(t,x) \in [T_3,\infty]\times\R$.}
\end{eqnarray}
From here, choosing    $x=m^-_{\beta}(t;u_0)$, we arrive at
$$
\beta\leq (1+\beta)\bar{\phi}_{c_*}(m^-_{\beta}(t;u_0)+c_*t+b')\leq  (1+\beta)\bar{\phi}_{c_*}(C+b')\quad\hbox{for all}\quad t>T_3,
$$
which contradicts the election of $\beta$.} 

\end{proof}

\bigskip

\begin{rem} 
By similar arguments, using the stability of semi--wavefronts in the non--local case (recently obtained by one of us \cite{S2}) one can show that  (a) and (b) also hold 
for the equation (\ref{nli}).
\end{rem}

\bigskip
\bigskip

\noindent
Notice that in the proof of the Proposition \ref{nac} the conclusions obtained depend strongly on the stability of the wavefronts. 
In fact, the main difficulty in the present case (in contrast with the situation without delay) is that the flow associated with
(\ref{nli}) in general is not monotone if  $g$ is not increasing. 

We are interested in obtaining information on the unboundness
in case (b)  of  Proposition \ref{nac} in the non local case. 
In that case, the possible asymmetry of the kernel might give place to a different set of admissible speeds for the semi--wavefronts in comparison with the symmetric case.
More precisely, it is well known that in the local case there is a  minimal speed $c_*>0$ (i.e., $c_*$ is the smallest positive real for which there exists
a nonnegative bounded solution of  
(\ref{nli}) of the form $u(t,x)=\phi_c(x+ct)$, $\phi_c:\R\to
\R$, satisfying   $\phi_c(-\infty)=0$). Moreover, for each $c\geq c_*$ we can consider the solutions
 $\psi_{-c}(x):=\phi_c(-x)$  as semi-wavefronts with speed $-c$. In that way we obtain a symmetric set of admissible speeds: $(-\infty,-c_*]\cup[c_*,+\infty)$. 

\bigskip
\noindent
Now, if we take the kernel $k_0$ satisfying,  

\vspace{3mm}
 
\noindent  
{\rm\bf($K_0$)} The kernel $k_0$ satisfies $k_0(\cdot)\geq 0, \quad\int_{\R}k_0(z)dz=1$ and, for given $a<0<b$ 
\begin{eqnarray*}
 \int_{\R}k_0(z)e^{\lambda z}dz<\infty \ \hbox{for all} \ \lambda\in(a,b).
\end{eqnarray*}
Then, by  \cite[Theorem 18]{GPT}, the set of admissible speeds is given by  $(-\infty, c_*^-]\cup[c_*^+,+\infty)$, where the speeds  
 $c_*^-<c_*^+$ are not necessarily opposed to each other (i.e., $c_*^- \neq - c_*^+$). Moreover, there exist values $\lambda_*^-<0<\lambda_*^+$ for which the curves 
\begin{eqnarray*}
f_1^{\pm}(z)=-z^2+c_*^{\pm}z+1\quad \hbox{and} \quad f^{\pm}_2(z)=g'(0)e^{-zc_*^{\pm}h}\int_{\R}k_0(y)e^{-zy}dy
\end{eqnarray*}
are tangent at $z=\lambda_*^{\pm}$ \cite[Lemma 22]{GPT}.

\bigskip
\noindent
The next result is a modest  generalization of the work of McKean \cite{Mc}. 

\begin{thm}\label{ail}
Let  $g$ satisfy {\rm\bf(M)}, $k_0$ satisfy  {\rm\bf($K_0$)}, and   $u_0$ be an initial data for  (\ref{nli}). Then we have, 

\begin{itemize}
\item[(i)]{If the initial datum satisfies  $e^{\lambda_*^+(\cdot)}u_0(\cdot)\in L^1(\R)\cap L^{\infty}(\R)$ } then there exists  $B\in\R$ such that
\begin{eqnarray}\label{McK}
 m_{\beta}^{-}(t,u_0)\geq  \frac{1}{2\lambda_*^{+}}\log(t)-c_*^{+}t+B\quad \hbox{for all} \ t>0.
\end{eqnarray}
\item[(ii)] If the initial datum satisfies  $e^{\lambda_*^-(\cdot)}u_0(\cdot)\in L^1(\R)\cap L^{\infty}(\R)$,  then there exists  $B\in\R$ such that
\begin{eqnarray}\label{McK2}
m_{\beta}^{+}(t,u_0)\leq  \frac{1}{2\lambda_*^{-}}\log(t)-c_*^{-}t+B \quad \hbox{for all} \ t>0.
\end{eqnarray}
\end{itemize}
\end{thm} 
{\begin{rem}[Logarithmic term]
In case the kernel $k_0$ is non symmetric it can happen that  $c_*^+< 0.$ If $u(t,\cdot)$ is asymptotically propagated with speed $-c_*^+$ (like backward traveling fronts, see, e.g.,  \cite[Page 16]{GPT}) the logarithmic term in
 (\ref{McK}) increases the speed $m_{\beta}^{-}(t,u_0)/t$, for large $t$, in constrast with the  local case \cite{Br}.
\end{rem}}

\begin{thm} \label{propnueva}

For a kernel $k_0$ satisfying {\rm\bf(K)} and $g$ satisfying  {\rm\bf(M)},   

\bigskip
\noindent
(i) Let $u(t,z)$ be the solution to (\ref{nli}) with initial data as in Theorem \ref{ail} (i). If $c \ge c_*^+$, then
$$
\lim_{t \to \infty, z \le - ct} u (t,z) =0. 
$$

\bigskip
\noindent
(ii) Let $u(t,z)$ be the solution to (\ref{nli}) with initial data as in Theorem \ref{ail} (ii). If $c \le c_*^-$, then
$$
\lim_{t \to \infty, z \ge  -ct} u (t,z) =0. 
$$

\bigskip
\noindent
In particular, if $c_*^+ c_*^- >0$ and the initial data has compact support then, 
$$
\lim_{t \to \infty} u (t,z) =0\quad \hbox{for all} \ z\in\R. 
$$
\end{thm}

\begin{rem}
Thus, in the case  $c_*^-c_*^+> 0$ 
the associated fronts extinguish in a different manner than in the local monotone case (but without assuming a KPP condition)
(see \cite[Proposition 1.3]{STR}) or the symmetric non local monotone case (assuming  the KPP condition) \cite[Theorem 4.1]{TZ}. 
However, in a very abstract context Liang and Zhao \cite[Theorem 3.4]{LZ} has obtained Theorem \ref{propnueva} for a monotone $g$ but for general $k_0$.
 \end{rem}

{\begin{rem}\label{spropa}
Note that if for some $\sigma>0$ the initial datum $u_0\in BUC(\R)$ (i.e., is bounded and uniformly continuous) satisfies $u_0\geq\sigma$ on a ball of radius $r_{\sigma}$ 
then, by using upper and lower nonlinearities given in \cite[Lemma 4.1]{YZ} along with \cite[Theorem 3.2]{YZ} we can conclude that for all $c> c_*^+$ and $c<c_*^-$
\begin{eqnarray}\label{eps}   
\lim_{t\to+\infty}\min_{ -c't\leq z\leq ct}u(t,z)\geq \epsilon_0,
\end{eqnarray}
for some $\epsilon_0>0.$ Moreover, the restriction on the radius  $r_{\sigma}$  can be dropped due to the fact that  the KPP condition implies the 
{\it sub homogenous} condition assumed in \cite[Theorem 3.4 part (2)]{LZ} and also because the  comparison Lemma \ref{alem}, the inequality (\ref{eps}) 
is valid for the class of exponentially bounded  initial data. Finally, notice that these propagation results admit a suitable interpretation in \cite{YZ} when the 
minimal wave speeds exist, i.e., when $g^2$ has only one positive fix point and $k$ satisfies the condition (K2) (see \cite[Theorem 4.4]{YZ}) 
while this limitation is overcome in \cite{GPT} by using the minimal  semi-wavefront speeds (see \cite[Theorem 18]{GPT}).    
\end{rem}}

\bigskip
\bigskip
\noindent
Theorem   \ref{ail} and Theorem \ref{propnueva}  follow from the following lemma.

\begin{lem}\label{flem}
 Let us consider  (\ref{dle0}) with
 $$
 m_{c_*^{\pm}}=2\lambda_*^{\pm} -c_*^{\pm}, \ q_{c_*^{\pm}}=-(\lambda_*^{\pm})^{2}+c_*^{\pm}\lambda_*^{\pm}+1,
 $$ 
 and
 $$
 k(x)=g'(0) e^{-\lambda_*^{\pm} c_*^{\pm}h} k_0(x+c_*^{\pm}h)e^{-\lambda_*^{\pm}\, x}
 $$
 If  $v^{\pm}(t,z)$ is the solution of (\ref{dle0}), with this choice of coefficients,  and with initial data  $v_0$ we have the following, 
 
 \bigskip
 \begin{itemize} 
\item[(i)]
If the initial data $v_0$ is as in Theorem \ref{ail} (i) then, 
\begin{eqnarray}\label{inqf}
u(t,z-c_*^{+}t)\leq e^{\lambda_*^+z}v^+(t,z)\quad \hbox{for all  $t>0, z\in\R$.}
\end{eqnarray}
Moreover, $v^+(t,z)$ satisfies (\ref{asym}) with $\gamma_m=0$ and $z_m=\lambda_*^+$.
 
\bigskip 
 
 \item[(ii)]
If the initial data $v_0$ is as in Theorem \ref{ail} (ii) then, 
\begin{eqnarray}\label{inqf1}
u(t,z-c_*^-t)\leq e^{\lambda_*^-z}v^-(t,z) \quad \hbox{for all  $t>0, z\in\R$.} 
\end{eqnarray}
 Moreover, $v^-(t,z)$ satisfies (\ref{asym}) with $\gamma_m=0$ and $z_m=\lambda_*^-$.
 \end{itemize}
 \end{lem}

\begin{proof} 
(i) Setting  $w(t,z):=e^{-\lambda_*^+z}u(t,z-c^+_*t)$ then, for all $(t,z)\in\R_+\times\R$, $w$ satisfies,
\begin{eqnarray*}\label{nla}
w_t(t,z)=w_{zz}(t,z)+m_{c_*^+}w_z(t,z)-q_{c_*^+}w(t,z)+e^{-\lambda_*^+z}\int_{\R}k_0(z-y)g(e^{\lambda_*^+(y-c_*^+h)}w(t-h,y-c_*^+h))dy.
\end{eqnarray*}
hence (\ref{inqf}) follows from Lemma \ref{alem}.

\noindent
{Finally, since $f_1^+(z)$ and $f_2(z)$ are tangent at $z=\lambda_*^{\pm}$ then by applying Theorem \ref{sba} one obtains (\ref{asym}) with $\gamma_m=0$ and $z_{m}=\lambda^{\pm}_{*}.$}
 
 \bigskip
 \noindent
(ii) This case is completely analogous to (i). 
\end{proof}

\vspace{2mm}

\begin{proof}[Proof  of Theorem \ref{ail}]

\bigskip
\noindent
(i) By using (\ref{inqf}) and (\ref{O}) we have
\begin{eqnarray}\label{rte}
 u(t,z-c_*^{+}t)\leq e^{\lambda_*^+z}v^+(t,z)\leq Ce^{\lambda_*^+(z-\frac{\log t}{2\lambda_*^+ })}\quad \hbox{for all  $t>0, z\in\R$.}
 \end{eqnarray}
and evaluating (\ref{rte}) at $z=c_*^+t+m^-_{\beta}(t)$ we obtain
\begin{eqnarray}\label{rta}
0<\beta\leq C e^{\lambda^+_*M(t)}\quad \hbox{for all} \ t>0
\end{eqnarray}
where $M(t)=c^+_*t+m^-_{\beta}(t)-\frac{\log t}{2\lambda^+_*}$. But, if there exist $\{t_n\}$ such that $t_n\to+\infty$ and $M(t_n)\to-\infty$ we obtain contradiction in (\ref{rta}) therefore $M(t)$ is bounded below which implies (\ref{McK}).

\bigskip
\noindent
(ii) Similarly to (i), by using (\ref{rta}) since $\lambda_*^-<0$ we conclude that $M^*(t):=c^-_*t+m^+_{\beta}(t)-\frac{\log t}{2\lambda^-_*}$ must be  bounded above. 
\end{proof}

\bigskip
\bigskip
We conclude this section with the proof of the Proposition \ref{propnueva}

\begin{proof}[Proof of Theorem \ref{propnueva}]

\bigskip
\noindent
(i) By (\ref{rte})

\begin{eqnarray}\label{rto}
 u(t,z)\leq Ce^{\lambda_*^+(z+c_*^+t)}\leq Ce^{(c_*^+-c)t}\quad \hbox{for all}  \ t>0, z\leq -ct.
 \end{eqnarray}
which implies the assertion. 

\bigskip
\noindent
The proof of (ii) is analogous.
\end{proof}
 
\section*{Acknowledgements}

It is a pleasure to thank Sergei Trofimchuk for many valuable discussions. 
The work of RB has been supported by Fondecyt (Chile) Projects  \# 114--1155, and  \# 116--0856. 
The work of AS has been supported by Fondecyt (Chile) Project \# 316--0473. We thank the anonymous referees
for their many insightful suggestions that helped us to improve the manuscript.

\bigskip

\section*{Appendix}

\bigskip
\noindent
We consider the nonlinear equation 
\begin{eqnarray}\label{dle1}
u_t(t,x) & =u_{xx}(t,x)+ mu_x(t,x)+pu(t,x)+\int_{\mathbb{R}} k_0(x-y) g(u (t-h, y))\, dy.
\end{eqnarray}
The following Lemma can be obtained from \cite[Proposition 3.1 and Lemma 5.6]{S2}.
\begin{lem}\label{alem}
Let us consider  $k_0$ satisfying {\bf (K)}  and $g$ satisfying the KPP condition $g(u)\leq g'(0)u$ for all $u\geq 0$. 
Denote by  $u(t,x)$  the solution to (\ref{dle1})  generated by some initial data $u_0$ and 
{$v(t,x)$ the solution (\ref{dle1}) with $g(u)=g'(0)u$ generated by $v_0$}. If  for some $\lambda\in (a,b)$ and $N>0$ 
  
   \begin{eqnarray}\label{vv}
 0\leq u_0(s,x)\leq v_0(s, x) \leq N e^{\lambda x}\ \hbox{for all}\ (s, x) \in  [-h, 0] \times\R.
 \end{eqnarray}
 then for each $n\in\Z_+$
 \begin{eqnarray}\label{estf}
 0\leq {u(t,x)\leq v(t, x)\leq N'\theta^{n} e^{\lambda x}\ \hbox{for all}\ (t, x) \in  [(n-1)h, nh] \times\R},
  \end{eqnarray} 
  for some $N'>0$ and $\theta>1$.
\end{lem}
\begin{proof}
 We consider $g$ only satisfying the KPP condition and $\lambda\in\R$. Then, by making the change of variables $\bar{u}(t,x):=u(t,x)e^{-\lambda x}$ the equation (\ref{dle1}) is transformed to
 \small\begin{eqnarray}\label{ale1}
\bar{u}_t(t,x)=\bar{u}_{xx}(t,x)+(2\lambda+m)\bar{u}(t,x)-q_1(\lambda)\bar{u}(t,x)+\int_{\R}k'(x-y)d(t,y)\bar{u}(t-h,y)dy, 
\end{eqnarray}\normalsize
where, $k'(y)=k_0(y)e^{-\lambda y}$ and ${d(t,y)=g(u(t-h,y))/u(t-h,y)}$. 
Next,  by the change of variable $\bar{\bar{u}}(t,x):=\bar{u}(t,x-(2\lambda+m) t)e^{q_1(\lambda) t}$ the equation (\ref{ale1}) is reduced
to the inhomogeneous heat equation, 
\begin{eqnarray}\label{He}
{\bar{\bar{u}}_t}(t,x)=\bar{\bar{u}}_{xx}(t,x)+f(t,x) \quad \hbox{for all} \quad (t,x)\in\R_+\times\R^d
\end{eqnarray}
where $$f(t,x)=e^{q_1(\lambda)h}\int_{\R^d}k'(y)d(t,x-(2\lambda+m) t-y)\bar{\bar{u}}(t-h,x-y-(2\lambda+m) h)dy.$$
{By (\ref{vv}) we get $\bar{u}(s,\cdot),\bar{\bar{u}}(s,\cdot)\in C([-h,0], L^{\infty}(\R))$ therefore $f(t,\cdot)\in C([0,h], L^{\infty}(\R))$. Next, } by denoting $\Gamma_t$ the one-dimensional heat kernel,  we have
\begin{eqnarray}\label{heat}
\bar{\bar{u}}(t)= \Gamma_t\ast \bar{\bar{u}}(0)+\int_0^t\Gamma_{t-s}\ast f(s,\cdot)ds
\end{eqnarray}
So that, for $t\in(0,h]$
\begin{eqnarray*}
||\bar{\bar{u}}(t)||_{L^{\infty}(\R)}&\leq& ||u(0)||_{L^{\infty}(\R)}+h\sup_{s\in[-h,0]}||f(s,\cdot)||_{L^{\infty}(\R)}\\
&\leq& (1+h{g'(0)e^{q_1(\lambda)h}||k'||_{L^1})||\bar{\bar{u}}_0||_{C([-h,0], L^{\infty}(\R))}},
\end{eqnarray*}
By defining $\theta_0:=1+he^{q_1(\lambda)h}||d_2||_{L^{\infty}}||k'||_{L^1}$ and repeating the argument with de initial data $u(h+s,x), u(2h+s,x)...$ we conclude
$$
\bar{\bar{u}}(nh+s, x)\leq {\theta_0^{n}} \ ||\bar{\bar{u}}_0||_{C([-h,0], L^{\infty})}\ \hbox{for all}\ (s, z) \in  [-h, 0] \times\R,
$$
which implies 
\begin{eqnarray}\label{exbo}
u(nh+s, x)\leq N'{\theta^{n}} e^{\lambda z}\ \hbox{for all}\ (s, z) \in  [-h, 0] \times\R,
\end{eqnarray}
with $N'=Ne^{2|q_1(\lambda)|h}\ \theta_0 $ and $\theta=\theta_0 \ {e^{q_1(\lambda)h}}$. Therefore, we conclude that $u(t,z)$ and $v(t,z)$ (for $v$ we take $g(v)=g(0)v$) are exponentially bounded for each $t\geq -h$ and uniformly exponentially bounded for $t$ on any compacts, i.e., for each $n\in\Z_+$ we have 
$$ 
u(t,x), v(t,x)\leq N'\theta^ne^{\lambda x}\quad \hbox{for all}\quad (t,x)\in[(n-1)h, nh]\times\R.
$$ 
Finally, by defining $\delta(t,x):=u(t,x)-v(t,x)$ and 
$$
(\mathcal{L}\delta) (t,x)=\delta_{xx}(t,x)- \delta_{t}(t,x)+m\delta_{x}(t,x)+p\delta(t,x)
$$
we have
\begin{eqnarray*}
\mathcal{L}\delta(t,x)&=&\int_{\R}k(y)[g'(0)v(t-h,x-y)-g(u(t-h,x-y))]dy\\
& \geq &\int_{\R}k(y)[g'(0)v(t-h,x-y)-g'(0)u(t-h,x-y)]dy\geq 0 
\end{eqnarray*}
for all $ t\in [0,h], \ x \in \R$, and (\ref{vv}) implies $\delta(0,x)\leq 0$ for $x\in\R$. So that, since by (\ref{exbo}) the function $\delta(t,x)$ is exponentially bounded then
Phragm\`en-Lindel\"of principle from \cite[Chapter 3, Theorem 10]{PW}  implies $\delta(t,x)
\leq 0$ for all  $t \in [0,h], \ x \in \R$. Then, by applying the same argument to $\delta(h+s,x), \delta(2h+s,x)...$ we conclude (\ref{estf}). 

\end{proof}

\end{document}